\documentclass[10pt]{article}

\usepackage{amsmath}
\usepackage{amssymb}
\usepackage{amsthm}

\newcommand{\R}{\mathbb{R}}
\newcommand{\C}{\mathbb{C}}
\newcommand{\Z}{\mathbb{Z}}
\newcommand{\N}{\mathbb{N}}
\newcommand{\cc}[1]{\overline{#1}}
\newcommand{\bs}{\boldsymbol}
\newcommand{\Tr}{\mathop{\mathrm{Tr}}}
\newcommand{\te}{\tilde}
\newcommand{\bk}{{\bs k}}
\newcommand{\bss}{{\bs s}}
\newcommand{\bxi}{{\bs\xi}}
\newcommand{\bx}{{\bs x}}
\newcommand{\bn}{{\bs n}}
\newcommand{\set}[2]{\left\{#1:#2\right\}}
\newcommand{\cA}{{\cal A}}
\newcommand{\cB}{{\cal B}}

\newcommand{\cG}{{\cal G}}
\newcommand{\cS}{{\cal S}}
\newcommand{\cD}{{\cal D}}
\newcommand{\cF}{{\cal F}}
\newcommand{\cV}{{\cal V}}
\newcommand{\cW}{{\cal W}}
\newcommand{\cL}{{\cal L}}

\newtheorem{theorem}{Theorem}[section]
\newtheorem{lemma}[theorem]{Lemma}
\newtheorem{corollary}[theorem]{Corollary}

\newtheorem{proposition}[theorem]{Proposition}

\newtheorem{remark}{Remark}[section]
\newtheorem{definition}{Definition}[section]

\numberwithin{equation}{section}

\begin{document}

\title{ELLIPTIC SCALING FUNCTIONS AS COMPACTLY SUPPORTED MULTIVARIATE ANALOGS OF THE B-SPLINES%
\thanks{To appear in IJWMIP}}

\author{Victor G.~Zakharov\\[0.5ex]
\small Institute of Continuum Mechanics of Russian Academy of Sciences,\\
\small              Perm, 614013, Russia\\
\small  E-mail: \tt    victor@icmm.ru
}

\maketitle{}

\begin{abstract}
In the paper, we present a family of multivariate compactly supported
scaling functions, which we
call as elliptic scaling functions. The elliptic scaling functions are the
convolution of elliptic splines, which correspond to homogeneous elliptic
differential operators, with distributions. The elliptic scaling functions
satisfy refinement relations with real isotropic dilation matrices.
The elliptic scaling functions satisfy most of the properties of the
univariate cardinal B-splines:
compact support, refinement relation, partition of unity,
total positivity, order of
approximation, convolution relation, Riesz basis formation (under
a restriction on the mask), etc.
The algebraic polynomials contained in
the span of integer shifts of any elliptic scaling function belong to
the null-space of a homogeneous elliptic differential operator.
Similarly to
the properties of the B-splines under differentiation, it is
possible to define elliptic (not necessarily differential) operators such
that the elliptic scaling functions satisfy relations with these operators.
In particular, the elliptic scaling functions can be considered as a
composition of segments, where the function inside a segment, like
a polynomial in the case of the B-splines, vanishes under the
action of the introduced operator.
\end{abstract}

\medskip



\section{Introduction}

The $d$-variate polyharmonic splines
(thin plate splines, etc.),
see for example Refs.~\cite{Duchon,HarderDesmarais,MadychNelson1,MadychNelson2},
can be presented as some linear combinations of shifted versions of the Green functions
of the polyharmonic operators
$\Delta^m$, $m=1,2,\dots$, ($\Delta:=\sum_{j=1}^d\frac{\partial^2}{\partial x_j^2}$ is the Laplace operator).
The Green function of the operator $\Delta^m$
is of the form
\begin{equation*}
  \rho(\bx) := \left\{
    \begin{aligned}
      &|\bx|^{2m-d}\log|\bx| & & \mbox{if $d$ is even},\\
      &|\bx|^{2m-d} & & \mbox{if $d$ is odd,}\qquad \bx\in\R^d;\\
    \end{aligned}\right.
\end{equation*}
and the Fourier transform of the Green function is
\begin{equation*}
  \hat\rho(\bxi) := \dfrac{1}{|\bxi|^{2m}},\qquad \bxi
         \in\R^d.
\end{equation*}
Here and in the sequel,
we denote vectors by boldface symbols and shall
not distinguish vectors as points of the Euclidean spaces and as column-matrices.

A particular case of the polyharmonic splines,
the so-called elementary polyharmonic cardinal B-splines (the term of Ch.~Rabut~\cite{Rabut},
see also Refs.~\cite{MicchelliRabutUtreras,Rabut1,VilleBluUnser})
are most similar to the univariate cardinal B-splines.
In the Fourier domain, the elementary polyharmonic cardinal B-spline is of the form
\begin{equation}\label{PolyharmonicBSpline}
  \hat B_m(\bxi) := \dfrac{|2\sin(\bxi/2)|^{2m}}{|\bxi|^{2m}}
      = \left(\dfrac{4\sum_{n=1}^d\sin^2(\xi_n/2)}{\sum_{n=1}^d\xi_n^2}\right)^{m},
\end{equation}
where $m\ge d/2$, $\bxi:=(\xi_1,\dots,\xi_d)$,
$\sin\bxi:=(\sin\xi_1,\dots,\sin\xi_d)$.
In the present paper, we shall say that functions~\eqref{PolyharmonicBSpline} are called {\em polyharmonic B-splines}.
The polyharmonic B-splines
satisfy most of the properties of the univariate cardinal B-splines,
see Refs.~\cite{Rabut,VilleBluUnser}.
However an important property of the univariate B-splines being compactly supported is violated.
(The polyharmonic B-splines decay like $O\left(1/|\bx|^{d+2}\right)$ as $|\bx|\to\infty$~\cite{Rabut}).
Note also that {\em any} (multivariate) polyharmonic spline (defined in
the Cartesian coordinate system) is not compactly supported. See, for example, Ref.~\cite{LorentzOswald}.
Furthermore, the refinement relation, another important property of the
univariate cardinal B-splines, does not hold also, see
Refs.~\cite{Madych90,Madych92,MicchelliRabutUtreras,VilleBluUnser}.

In the paper, we present multivariate compactly-supported scaling functions.
For any real isotropic dilation matrix,
we can construct a trigonometric polynomial mask such that
the Fourier transform of the corresponding scaling function is of the form
\begin{equation}\label{ScalingFunction}
  \hat\phi^m(\bxi) := \dfrac{{G}(\bxi)}{(P(\bxi))^m}{M}(\bxi),\qquad m\ge1,\quad \bxi\in\R^d,
\end{equation}
where $P(\bxi)$ is a positive definite quadratic form (in particular, $|\bxi|^2$),
${G}(\bxi)$ is a trigonometric polynomial,
and the function ${M}(\bxi)$ is continuous on $\R^d$ and
does not decay at the infinity.
Thus scaling function~\eqref{ScalingFunction} can be considered as the convolution of
an elliptic spline, whose Fourier transform is $\dfrac{{G}(\bxi)}{(P(\bxi))^m}$,
with a distribution.
In the paper,
we shall say that the scaling functions of the form~\eqref{ScalingFunction} are called
{\em elliptic scaling functions}.

We state that any {\em real isotropic} (not necessarily dilation)
matrix $A$
is similar,
with a positive definite similarity transformation matrix, to an
orthogonal matrix; and the similarity transformation matrix
defines the positive definite quadratic form $P(\bxi)$.
Note that the quadratic form $P(\bxi)$ (and the corresponding homogeneous elliptic operator)
is invariant (up to the constant factor) under the coordinate transformation by the matrix $A$.
The invariance of the quadratic form $P(\bxi)$ is crucial for the properties of
scaling function~\eqref{ScalingFunction}.

In the paper, we show that elliptic scaling functions~\eqref{ScalingFunction} have many properties of the
univariate B-splines
and multivariate polyharmonic B-splines.
Any elliptic scaling function $\phi^m$ partitions unity,
is totally positive, its order of approximation is $2m$.
We present a sufficient condition that the integer shifts of $\phi^m$ form a Riesz basis.
Also note that the algebraic polynomials reproduced by the scaling function $\phi^m$ belong to the
null-space of the elliptic operator whose symbol is $(P(\xi))^m$.

Unlike the polyharmonic B-splines and like the univariate B-splines,
the elliptic scaling functions
are compactly supported and, by definition, satisfy refinement relations.

Considering the function $\widehat{\Delta^\sharp}(\bxi):=\dfrac{P(\bxi)}{\left({M}(\bxi)\right)^{1/m}}$,
see~\eqref{ScalingFunction},
as the symbol of an
operator, we see that the elliptic scaling function $\phi^m$,
similarly to the B-splines,
satisfies
relations with this operator. In particular, we have:
$\left(\Delta^\sharp\right)^m \phi^m(\bx) = 0$, $\forall \bx\in\R^d\setminus\Z^d$.
Thus the elliptic scaling functions
can be considered as a composition of segments,
where the function inside a segment, like a polynomial in the case of the B-splines,
vanishes under the action of the operator $\left(\Delta^\sharp\right)^m$.
Note also that the elliptic scaling functions $\phi^m$, in the Fourier domain, can be written
as $\hat\phi^m(\bxi)=\left(\dfrac{G(\bxi)}{\widehat{\Delta^\sharp}(\bxi)}\right)^m$,
where $G$ is a trigonometric polynomial
and $\widehat{\Delta^\sharp}$ is the symbol of an elliptic operator.
Hence we can say that the elliptic scaling functions can be considered as multivariate analogs of the
univariate cardinal B-splines.
On the other hand,
we shall see that the univariate cardinal B-splines come under the presented approach.
Namely the univariate cardinal B-splines of odd degree are a particular case of elliptic
scaling functions~\eqref{ScalingFunction}.

Note that, in the paper, we consider only the scaling
functions. The construction of the corresponding wavelets is not discussed.
Moreover, the orthogonalization or construction of dual
scaling functions (in the biorthogonal case)
are not discussed also. This
will be the object of another paper.

The paper is organized as follows. In Section~\ref{Section_Similarity},
we present some properties of the isotropic matrices. In particular, the similarity to orthogonal matrices,
the positive definite quadratic forms defined by the similarity transformation matrices,
and the invariance of the quadratic forms are
considered.
Section~\ref{Section_MaskChoice} is devoted to the construction of compactly supported
elliptic scaling functions of the first order.
The explicit form of the Fourier transform of the scaling functions is derived.
In this section, some properties of the first order scaling functions are presented.
Section~\ref{Section_ArbitraryOrderFunctions} is devoted to the definition
and properties of the elliptic scaling
functions of an arbitrary order.
In Section~\ref{Section_Examples},
several bivariate isotropic dilation matrices are presented and
some properties of the corresponding elliptic scaling functions
are discussed.

\subsection{Preliminaries and notations}

Let us introduce here some general notation.

For any vector $\bs v\in\R^d$, by $|\bs v|:=\sqrt{v_1^2+\cdots+v_d^2}$ denote vector's {\em length};
in other words, $|\bs v|$ is the usual Euclidean norm of $\bs v$.

A {\em multi-index} $k$ is a $d$-tuple
$(k_1,\dots,k_d)$ with its components being nonnegative
integers, i.\,e., $k\in\Z^d_{\ge0}$.
The {\em length} of the multi-index $k$ is $|k|:=
k_1+\cdots+k_d$.

By $\bx^k$ denote the monomial $x_1^{k_1}\cdots
x_d^{k_d}$, where $x:=(x_1,\dots,x_d)\in\R^d$,
$k:=(k_1,\dots,k_d)\in\Z^d_{\ge0}$. Note that the
{\em total degree} of $\bx^k$ is $|k|$.

Let $\Pi$ denote the space of {\em all} polynomials on $\R^d$.
Also by $\Pi_N$, $N\in\Z_{\ge0}$, denote the space of polynomials with
total degree less than or equal to $N$.

By $D^{k}$, $k\in\Z^d_{\ge0}$, denote the differential operator
$D_1^{k_1}\cdots D_d^{k_d}$, where $D_n$, $n=1,\dots,d$,
is the partial derivative with respect to the $n$th coordinate.

By $\bs e_j$ denote the $j$th basis vector of the space $\R^d$, i.\,e.,
$\bs e_j:=(\delta_{j1},\dots,\delta_{jd})$, where
$\delta_{jk}
 :=\left\{\begin{array}{ll}
    1, & \text{$j=k$,} \\
    0, & \text{$j\ne k$.}
  \end{array}\right.$

The Fourier transform of a function $f\in L^1(\R^d)$,
$d=1,2,\dots$, is defined as
$$
  {\cF}(f)(\bxi)=\hat f(\bxi):=\int_{\R^d} f(\bx)e^{-i\bxi\cdot \bx}\,d\bx,\qquad
     \bx,\bxi\in\R^d,
$$
where $\bx\cdot\bxi:=x_1\xi_1+\cdots+x_d \xi_d$.

By $\mathcal{S}'$ denote the {\em space of tempered distributions} and
by $\delta$ denote the {\em Dirac distribution}.

In the paper,
we shall suppose that the {\em dilation matrices} are
real integer matrices whose eigenvalues are greater
than 1 in absolute value.
Then, for any dilation matrix $A$, we have the following property
\begin{equation}\label{DilationMatrixProperty}
  \quad\lim_{j\to\infty}\left|A^{j}\bs x\right|\to\infty,\qquad \forall\bs x\in\R^d,\ \bs x\ne\bs 0.
\end{equation}

In the paper, we shall also suppose that the matrix-vector multiplication
is (left-)distributive over a (finite or countable) set of vectors of
an Euclidian space: $A\set{\bss_j}{j\in J}:=\set{A\bss_j}{j\in J}$,
where $A$ is a $d\times d$ matrix
and $\bss_j:=(s_{j,1},\dots,s_{j,d})\in\R^d$. Moreover, the distributive property can be extended to
uncountable sets.

To simplify the notations,
by $A^{-T}$ we shall denote the matrix
$(A^T)^{-1}\equiv(A^{-1})^T$.

A scaling function $\phi$ satisfies a {\em refinement relation}
\begin{equation}\label{ScalingRelation}
  \phi(\bx) = \sum_{\bk\in{\Z}^d}h_{\bk}|\det A|^\frac{1}{2}\phi(A\bx - \bk),
     \qquad\bx\in\R^d,
\end{equation}
where $A$ is a $d\times d$ dilation matrix.
Refinement relation~\eqref{ScalingRelation}
can be rewritten in the Fourier domain as
\begin{equation*}
  \hat\phi(\bxi) = m_0\left(A^{-T}\bxi\right)
     \hat\phi\left(A^{-T}\bxi\right),     \qquad \bxi\in\R^d,
\end{equation*}
where $m_0(\bxi)$, $\bxi\in\R^d$, is a $2\pi$-periodic function,
which is called {\em mask}. The Fourier transform of the scaling
function $\phi$ can be defined by the mask $m_0$ as follows
\begin{equation}\label{HatPhiInfiniteProduct}
  \hat\phi(\bxi)=\prod_{j=1}^\infty
    m_0\left(\left(A^{-T}\right)^{j}\bxi\right).
\end{equation}

\subsection{Auxiliary propositions on dilation matrices}

In this subsection, we present some well-known propositions and definitions concerning the dilation matrices.
\begin{proposition}\label{PropositionOnDigits}
Let $A$ be a non-singular $d\times d$ matrix with integer entries. Then the number
of the cosets of $\Z^d$ by modulo $A$ is equal to $|\det A|$ and
the set $\Z^d\cap A[0,1)^d$ is a set of representatives of the quotient $\Z^d/A\Z^d$.
\end{proposition}

\begin{definition}\label{Definition_SetOfDigits}
A full set of representatives of the
quotient $\Z^d/A\Z^d$ is called the {\em set of digits} of the matrix $A$; and we shall denote the set
of digits of the matrix $A$ by ${\cW}(A)$.
By ${\cS}(A)\subset[0,1)^d$ we denote the set of points such that $A{\cS}(A)={\cW}(A)$.
\end{definition}

\begin{remark}
By definition, the set of digits of the matrix $A$ does not depend on any specific choice
of the representatives of the quotient $\Z^d/A\Z^d$. Thus the sets $-{\cW}(A)$ and
$\bk+{\cW}(A)$, for some $\bk\in\Z^d$, contain different from the initial set ${\cW}(A)$ elements;
nevertheless they can be considered  (and denoted) as the set of digits ${\cW}(A)$.
On the other hand,
the elements of the set ${\cS}(A)$ are defined uniquely (up to a permutation).
\end{remark}

\begin{proposition}\label{LemmaDigitsProperties}
Let $A$ be a $d\times d$ nonsingular matrix with integer entries and ${\cW}(A)$
be its set of digits, then we have
\begin{equation*}
  \begin{aligned}
    &\bigcup_{\bn\in{\cW}(A)}\set{A\bk+\bn}{\bk\in\Z^d}=\Z^d;\\
    &\set{A\bk+\bn}{\bk\in\Z^d}\cap\set{A\bk+\bn'}{\bk\in\Z^d}=\emptyset,\quad \bn,\bn'\in{\cW}(A),\ \bn\ne \bn'.
  \end{aligned}
\end{equation*}
\end{proposition}
The proof of Propositions~\ref{PropositionOnDigits},~\ref{LemmaDigitsProperties}
can be found, for example, in the book~\cite{NPS}.

\begin{proposition}\label{PropositionMyDigit}
Let $A$ be a $d\times d$ {\em dilation} matrix with integer entries,
then we have
\begin{align*}
  &\bigcup_{j=1}^\infty\bigcup_{\bss\in\cS(A)\setminus\{\bs 0\}} A^j\set{\bk+\bss}{\bk\in\Z^d} = \Z^d\setminus\{\bs 0\},\\
  &A^j\set{\bk+\bss}{\bk\in\Z^d} \cap A^{j'}\set{\bk+\bss'}{\bk\in\Z^d}=\emptyset\qquad
    \mbox{if }j\ne j'\mbox{ or } \bss\ne\bss',
\end{align*}
where the set ${\cS}(A)$ is defined in Definition~\ref{Definition_SetOfDigits}.
\end{proposition}
To prove Proposition~\ref{PropositionMyDigit}, we can refer the reader
to Ref.~\cite{Z_ConstrApprox_NDStrFix}.


\section{Similarity to Orthogonal Matrices}\label{Section_Similarity}


Let us recall the definition of isotropic (dilation) matrices.
\begin{definition}
A square matrix is called {\em isotropic} if the matrix is
diagonalizable over $\C$ and all its eigenvalues are equal in absolute value.
\end{definition}

\begin{theorem}\label{Theorem_MyDecomposition}
Let $\te A$ be a square non-singular real matrix.
Suppose the matrix $\te A$ is diagonalizable over $\C$ and all its eigenvalues
are equal to $1$ in absolute value,
then
\begin{equation}\label{MyDecomposition}
  \te A=QUQ^{-1},
\end{equation}
where $U$ is an orthogonal (real)
matrix and $Q$ is a symmetric positive definite
(real) matrix.
\end{theorem}
To prove the theorem,
we can refer the reader to Refs.~\cite{Z_ConstrApprox_NDStrFix,Z_SIAM_DilMatrs}.
Here note only that the proof is based on the diagonalizability of
the matrix $\te A$
and the polar decomposition
of the
similarity transformation matrix.

So we see that any real isotropic (dilation) matrix is similar (up to a constant factor)
to an orthogonal matrix.

The next corollary directly follows from Theorem~\ref{Theorem_MyDecomposition} and
will be very useful hereinafter.

\begin{corollary}\label{CorollaryFromMyDecomposition}
From Theorem~\ref{Theorem_MyDecomposition} it follows
\begin{align}
    \label{Q-2Invariance}
    &\te AQ^{2}\te A^T=\te A^{-1}Q^{2}\te A^{-T}=Q^{2},\\
    \label{Q2Invariance}
    &\te A^TQ^{-2}\te A=\te A^{-T}Q^{-2}\te A^{-1}=Q^{-2}.
\end{align}
\end{corollary}
Using~\eqref{MyDecomposition}, the proof is straightforward.

Consider a real square matrix $A$ with integer entries. Suppose that $A$ is isotropic;
then, using Theorem~\ref{Theorem_MyDecomposition},
$A^{-T}$ can be factored as follows
\begin{equation}\label{MatrixBDecomposition}
    A^{-T}= \frac{1}{q^{1/d}} Q^{-1}UQ,
\end{equation}
where $q:=|\det A|  $, $U$ is an orthogonal matrix,
and $Q$ is a symmetric positive definite matrix.
Now we can define a quadratic form
\begin{equation}\label{QuadraticForm}
  P(\bs x):={\bs x}^TQ^{2}\bs x,\quad \bs x\in\R^d.
\end{equation}
Since $Q^{2}$
is positive definite; therefore,
quadratic form~\eqref{QuadraticForm}
is positive definite.
By Corollary~\ref{CorollaryFromMyDecomposition}, we see that the quadratic
form $P(\bx)$
is invariant (up to the constant value)
under the variable transformation
by the matrix $A^{-T}$: $\bs x\mapsto\bs x':= A^{-T}\bs x$.
Indeed, using~\eqref{Q-2Invariance}, we have
\begin{equation}\label{QuadrFormInvariance}
  P(\bs x')=P(A^{-T}\bs x)=\bs x^T A^{-1} Q^{2}A^{-T}\bs x = \frac{1}{q^{2/d}}\bx^TQ^{2}\bx
  = \frac{1}{q^{2/d}}P(\bs x).
\end{equation}
(Similarly, the quadratic form ${\bs x}^TQ^{-2}\bs x$ will be invariant under the transformation
by the matrix $A$, see~\eqref{Q2Invariance}.)

\begin{remark}
Note that the differential operator corresponding to quadratic form~\eqref{QuadraticForm},
i.\,e., $P(\bxi)$ is the symbol of the operator, will be invariant under the coordinate transformation
by the matrix $A$.
\end{remark}

\begin{remark}
The matrix $Q$ in formulas~\eqref{MyDecomposition},~\eqref{MatrixBDecomposition}
(consequently, quadratic form~\eqref{QuadraticForm}) is
defined within a constant factor.
\end{remark}

\begin{remark}
We conjecture that, if an isotropic matrix $A$
is a matrix with integer entries; then, multiplying by
the appropriate real value, the matrix $Q^2$ that corresponds to the matrix
$Q$ in formulas~\eqref{MyDecomposition},~\eqref{MatrixBDecomposition}
can be made a matrix with integer entries. This will be discussed elsewhere.
\end{remark}


\section{Elliptic Scaling Function of the First Order}\label{Section_MaskChoice}


\subsection{Construction of the mask}

Let $A$ be an isotropic dilation matrix and let $A^{-T}$ be factored by
formula~\eqref{MatrixBDecomposition}.
Define a trigonometric function ${G}(\bxi)$ such that its Taylor series about zero begins
with quadratic form~\eqref{QuadraticForm},
i.\,e.,
\begin{equation}\label{m_1SimP}
  {G}(\bxi) := P(\bxi)+\mbox{ higher order terms}, \qquad\bxi\in\R^d.
\end{equation}

Define the mask $m_0$ as follows
\begin{equation}\label{m_0Definition}
  m_0(\bxi):=\dfrac{\prod\limits_{\bss\in{\cS}(A^T)\setminus\{\bs 0\}}{G}(\bxi+2\pi \bss)}
     {\prod\limits_{\bss\in{\cS}(A^T)\setminus\{\bs 0\}}{G}(2\pi \bss)}.
\end{equation}
(In formula~\eqref{m_0Definition}, we suppose that
$G(2\pi s)\ne0$, $\forall s\in\cS(A^T)\setminus\{0\}$.)

Let the matrix $Q^2$ be presented in component-wise form as follows
\begin{equation*}
  Q^2 :=
  \begin{pmatrix}
    q_{11} & q_{12} &\cdots & q_{1d} \\
    q_{12} & q_{22} &\cdots & q_{2d} \\
    \hdotsfor{4}\\
    q_{1d} & q_{2d} &\cdots & q_{dd} \\
  \end{pmatrix},\qquad
q_{ij}\in\R,\ \
i,j=1,\dots,d,\ i\le j;
\end{equation*}
and let $\bs\xi:=(\xi_1,\dots,\xi_d)$. Then quadratic form~\eqref{QuadraticForm} is
\begin{equation}\label{AnotherQuadraticForm}
  P(\bs\xi)
      :=\sum_{1\le i\le d} q_{ii}\xi_i^2+2\sum_{\begin{subarray}{c}
          1\le i,j\le d\\[0.4ex]i<j \end{subarray}} q_{ij}\xi_i\xi_j.
\end{equation}
It is easy to see
that the following trigonometric polynomial
has the required Taylor expansion about zero, see~\eqref{m_1SimP},
\begin{equation}\label{m_1}
  {G}(\xi_1,\dots,\xi_d) := 4\sum_{1\le i\le d} q_{ii}\sin^2\frac{\xi_i}2
    +2\sum_{\begin{subarray}{c}
          1\le i,j\le d\\[0.4ex]i<j \end{subarray}}q_{ij}\sin\xi_i\sin\xi_j.
\end{equation}
Thus, using~\eqref{m_1}, the mask $m_0$ given by~\eqref{m_0Definition} is a trigonometric polynomial.


\subsection{Explicit form of the Fourier transform of the scaling function}


Let $m_0(\bxi)$ be given by~\eqref{m_0Definition}. Acting
similarly to the classical univariate formula
$\prod_{j=1}^\infty\cos\left(\dfrac{\xi}{2^j}\right)=\dfrac{\sin\xi}{\xi}$,
see the book~\cite{Dau}, we can write
\begin{multline}\label{LastExpression}
  \prod_{j=1}^J m_0\left(\left(A^{-T}\right)^j\bs\xi\right)
    = \prod_{j=1}^J \frac{m_0\left(\left(A^{-T}\right)^j\bs\xi\right)
               {G}\left(\left(A^{-T}\right)^j\bs\xi\right)}
        {{G}\left(\left(A^{-T}\right)^j\bs\xi\right)}\\
           = \frac{{G}(\bs\xi)}{{G}\left(\left(A^{-T}\right)^J\bs\xi\right)}\prod_{j=0}^{J-1}
                 \frac{m_0\left(\left(A^{-T}\right)^{j+1}\bs\xi\right)
                          {G}\left(\left(A^{-T}\right)^{j+1}\bs\xi\right)}
                              {{G}\left(\left(A^{-T}\right)^{j}\bs\xi\right)},
\end{multline}
where ${G}(\bxi)$ is the same function as in formula~\eqref{m_0Definition}.
Unfortunately, unlike the univariate case, the fractions under the product sign
in~\eqref{LastExpression} are not canceled.
Denoting
\begin{equation}\label{F}
  \mu(\bs\xi):=\frac{q^{2/d}m_0\left(A^{-T}\bs\xi\right){G}\left(A^{-T}\bs\xi\right)}{{G}(\bs\xi)},
\end{equation}
where $q:=|\det A|$,
we have
\begin{equation}\label{Product}
   \prod_{j=1}^J m_0\left(\left(A^{-T}\right)^j\bs\xi\right)
     = \frac{{G}(\bxi)}{q^{2J/d}{G}\left(\left(A^{-T}\right)^J\bs\xi\right)}\prod_{j=0}^{J-1}
             \mu\left(\left(A^{-T}\right)^j\bs\xi\right).
\end{equation}
Using~\eqref{QuadrFormInvariance}, \eqref{m_1SimP}, we get
the limit of~\eqref{Product} as $J\to\infty$:
\begin{equation}\label{HatPhi1}
  \hat\phi(\bxi) = \frac{{G}(\bs\xi)}{P(\bxi)}{M}(\bs\xi),
\end{equation}
where
\begin{equation}\label{CalF}
  {M}(\bs\xi):=\prod_{j=0}^{\infty}\mu\left(\left(A^{-T}\right)^j\bs\xi\right).
\end{equation}

Note that, in this subsection, we implicitly suppose that the mask
$m_0$ is not necessarily a trigonometric polynomial, but we
suppose that the infinite product in the right-hand side
of~\eqref{HatPhiInfiniteProduct} converges almost everywhere.
However, in the sequel, we shall suppose that $m_0$ is a
trigonometric polynomial, i.\,e., $G$ is given by~\eqref{m_1}.


\subsection{Properties of the scaling function}



\subsubsection{Compact support}


We recall that, for any $d\times d$ dilation matrix and trigonometric polynomial mask,
there exists a unique up to a constant factor {\em compactly supported} solution $\phi\in S'(\R^d)$
of refinement relation~\eqref{ScalingRelation},
see Ref.~\cite{CDM}. Thus we have the following proposition.
\begin{proposition}
For any real isotropic dilation matrix and
mask~\eqref{m_0Definition}, where $G$ is given by
formula~\eqref{m_1}, the corresponding elliptic scaling function,
whose Fourier transform is of the form~\eqref{HatPhi1}, is
compactly supported.
\end{proposition}

In formula~\eqref{HatPhi1}, the Fourier transform of the elliptic spline is rather simple function.
Thus, in the next subsection, we give our attention to the function $M(\bxi)$. In particular, the behavior
of $M(\bxi)$ at the infinity determines the decay of $\hat\phi(\bxi)$ as $|\bxi|\to\infty$.


\subsubsection{Properties of $M(\bxi)$}


Below we present a lemma about the positive definiteness of the trigonometric polynomial $G$.

\begin{lemma}\label{LemmaZerosG}
  For any quadratic positive definite form~\eqref{AnotherQuadraticForm}, trigonometric
polynomial~\eqref{m_1}
is not negative on $\R^d$ and vanishes only at the points $2\pi \bk$, $\bk\in\Z^d$.
\end{lemma}
\begin{proof}

Rewrite formula~\eqref{m_1} as follows
$$
  {G}(\xi_1,\dots,\xi_d) = 4\sum_{1\le i\le d} q_{ii}\sin^2\frac{\xi_i}2
     + 8\sum_{\begin{subarray}{c}
          1\le i,j\le d\\[0.4ex]i<j \end{subarray}}q_{ij}
              \sin\frac{\xi_i}2\sin\frac{\xi_j}2\cos\frac{\xi_i}2\cos\frac{\xi_j}2.
$$
Since the quadratic form $P(\bxi)$
is positive definite; we have
\begin{equation*}
  \sum_{1\le i\le d} q_{ii}\sin^2\frac{\xi_i}2
     + 2\sum_{\begin{subarray}{c}
          1\le i,j\le d\\[0.4ex]i<j \end{subarray}}q_{ij}
              \sin\frac{\xi_i}2\sin\frac{\xi_j}2\equiv
                  P\left(\sin\frac{\xi_1}2,\dots,\sin\frac{\xi_d}2\right)\ge0,\quad
                     \forall\bxi\in\R^d,
\end{equation*}
and the trigonometric polynomial $P\left(\sin\dfrac{\xi_1}2,\dots,\sin\dfrac{\xi_d}2\right)$
vanishes iff $\sin\dfrac{\xi_j}2=0,\ j=1,\dots,d$.
Thus, since $0\le\cos\xi_j/2\le1$ for $\xi_j\in[-\pi,\pi]$, $j=1,\dots,d$, and
$G(\bxi)$ is $2\pi$-periodic; it follows that ${G}(\bxi)\ge0$ for all $\bxi\in\R^d$
and ${G}(\bxi)$ vanishes only at the points $2\pi \bk$, $\bk\in\Z^d$.
\end{proof}

It is convenient to rewrite formula~\eqref{m_0Definition} as follows
\begin{equation}\label{m0AnotherForm}
  m_0(\bxi):=\dfrac{\prod\limits_{\bn\in{\cW}(A^T)\setminus\{\bs 0\}}{G}(\bxi+2\pi A^{-T}\bn)}
     {\prod\limits_{\bn\in{\cW}(A^T)\setminus\{\bs 0\}}{G}(2\pi A^{-T}\bn)}.
\end{equation}
Now we can state and prove lemmas about some properties of the function $\mu(\bxi)$.
\begin{lemma}\label{Lemma_MuBounded}
  Let $\mu(\bs\xi)$ be given by~\eqref{F}, where $G(\bxi)$ is given by formula~\eqref{m_1}.
Then $\mu(\bxi)$ is $2\pi$-periodic and
\begin{equation}\label{Mu=1}
  \mu(2\pi\bs k)=1,\qquad \forall\bs k\in\Z^d.
\end{equation}
Moreover, $\mu(\bxi)$ is continuous on $\R^d$; and
there exist constants ${\cA}>0$ and $1\le{\cB}<\infty$ such that ${\cA}\le\mu(\bs\xi)\le{\cB}$,
$\forall\bxi\in\R^d$.
\end{lemma}
\begin{proof}
First, using~\eqref{m0AnotherForm},
for any $\bk\in\Z^d$, we have
\begin{multline*}
  \mu(\bxi+2\pi \bk)=
      \frac{q^{2/d}\prod\limits_{\bn\in{\cW}(A^T)}{G}(A^{-T}(\bxi+2\pi\bk)+2\pi A^{-T}\bn)}
         {G(\bxi+2\pi \bk)\prod\limits_{\bn\in{\cW}(A^T)\setminus\{\bs 0\}}
            {G}(2\pi A^{-T}\bn)}\\
            =\frac{q^{2/d}\prod\limits_{\bn'\in{\cW}(A^T)}{G}(A^{-T}\bxi+2\pi A^{-T}\bn')}
              {G(\bxi)\prod\limits_{\bn\in{\cW}(A^T)\setminus\{\bs 0\}}
                 {G}(2\pi A^{-T}\bn)}=\mu(\bxi)
\end{multline*}
\nopagebreak
(where $\bn'=\bk+\bn$).

Secondly consider the numerator of the fraction in the right-hand side of~\eqref{F}:
\begin{equation}\label{NumeratorMu}
  \prod\limits_{\bn\in{\cW}(A^T)}{G}(A^{-T}\bxi+2\pi A^{-T}\bn).
\end{equation}
Since ${G}(\bxi)$ is $2\pi$-periodic
and vanishes only at the points $2\pi\Z^d$, it follows that the previous expression vanishes
only at the points:
$2\pi(A^T\bk+\bn)$,  $\forall \bk\in\Z^d$ and $\forall\bn\in{\cW}(A^T)$.
Using Proposition~\ref{LemmaDigitsProperties}, we see that numerator~\eqref{NumeratorMu}
vanishes only at the points $2\pi \bk$, $\bk\in\Z^d$; and the zeros of the numerator do not superimpose
by different multipliers under the product sign.

Thirdly, using~\eqref{QuadrFormInvariance},~\eqref{m_1SimP},
since $\mu(\bxi)$ is $2\pi$-periodic, $m_0(0)=1$, and the zeros of the numerator
have the same multiplicity;
we have, for all $\bk\in\Z^d$,
$$
  \mu(2\pi\bk)=\lim_{\bxi\to\bs 0}\frac{q^{2/d}{G}(A^{-T}\bxi)}{{G}(\bxi)}
            =\lim_{\bxi\to\bs 0}\frac{q^{2/d}P(A^{-T}\bxi)}{P(\bxi)}=1.
$$

Finally, since $\mu(\bxi)$ is continuous, $2\pi$-periodic, and vanishes nowhere;
the condition $0<{\cA}\le\mu(\bs\xi)\le{\cB}<\infty$, $\forall\bxi\in\R^d$, is obvious.
\end{proof}

\begin{lemma}
The mask $m_0(\bxi)$ given by~\eqref{m_0Definition}
(or~\eqref{m0AnotherForm}), where $G(\bxi)$ is given
by~\eqref{m_1}, is even; i.\,e.,
$m_0(-\xi_1,\dots,-\xi_d)=m_0(\xi_1,\dots,\xi_d)$,
$\forall\bxi:=(\xi_1,\dots,\xi_d)\in\R^d$.
\end{lemma}
\begin{proof}
  Since polynomial~\eqref{m_1} is even, we have
\begin{multline*}
  m_0(-\bxi)=\dfrac{\prod\limits_{\bn\in{\cW}(A^T)\setminus\{\bs 0\}}{G}(-\bxi+2\pi A^{-T}\bn)}
     {\prod\limits_{\bn\in{\cW}(A^T)\setminus\{\bs 0\}}{G}(2\pi A^{-T}\bn)}\\
     =\dfrac{\prod\limits_{\bn'\in{\cW}(A^T)\setminus\{\bs 0\}}{G}(\bxi+2\pi A^{-T}\bn')}
     {\prod\limits_{\bn\in{\cW}(A^T)\setminus\{\bs 0\}}{G}(2\pi A^{-T}\bn)}=m_0(\bxi)
\end{multline*}
(where $\bn'=-\bn$).
\end{proof}

\begin{lemma}
  For function $\mu(\bxi)$ given by~\eqref{F},
where $m_0(\bxi)$ given by~\eqref{m_0Definition} and
$G(\bxi)$ by~\eqref{m_1},
the following estimations are valid
\begin{equation}\label{MuEstimations}
  1-C'P(\bxi) \le {\cal\mu}(\bxi)\le 1+C''P(\bxi),\quad C',C''>0,\ \forall\bxi\in\R^d,
\end{equation}
where $P(\bxi)$ is given by~\eqref{AnotherQuadraticForm}.
\end{lemma}
\begin{proof}
 The Taylor series expansion for trigonometric polynomial~\eqref{m_1} about zero is of the form
$$
  {G}(\bxi) = P(\bxi)+\sum_{\begin{subarray}{c}
          k\in\Z^d_{\ge0}\\[0.4ex]|k|=4,6,8,\dots \end{subarray}}a_k\bxi^k,\qquad a_k\in\R,\ \bxi\in\R^d.
$$
Since $m_0(\bxi)$ is an even function, its Taylor series about zero includes only even powers:
$$
  m_0(\bxi) = 1+\sum_{\begin{subarray}{c}
          k\in\Z^d_{\ge0}\\[0.4ex]|k|=2,4,6,\dots \end{subarray}}
             b_k\bxi^k,\qquad b_k\in\R,\ \bxi\in\R^d.
$$
Using property~\eqref{QuadrFormInvariance}, we have
\begin{multline}\label{MuExpansion}
    \mu(\bs\xi)=\dfrac{q^{2/d}\left(1+\sum\limits_{\begin{subarray}{c}
          k\in\Z^d_{\ge0}\\[0.4ex]|k|=2,4,6,\dots \end{subarray}}
             b_k\left(A^{-T}\bxi\right)^k
    \right)\left(P(A^{-T}\bxi)+\sum\limits_{\begin{subarray}{c}
          k\in\Z^d_{\ge0}\\[0.4ex]|k|=4,6,8,\dots \end{subarray}}a_k
            \left(A^{-T}\bxi\right)^k\right)}{P(\bxi)
              +\sum\limits_{\begin{subarray}{c}
          k\in\Z^d_{\ge0}\\[0.4ex]|k|=4,6,8,\dots \end{subarray}}a_k\bxi^k}\\[2ex]
          =\dfrac{
      P(\bxi)+\sum\limits_{\begin{subarray}{c}
          k\in\Z^d_{\ge0}\\[0.4ex]|k|=4,6,8,\dots \end{subarray}}a_k'
            \bxi^k}{P(\bxi)
              +\sum\limits_{\begin{subarray}{c}
          k\in\Z^d_{\ge0}\\[0.4ex]|k|=4,6,8,\dots \end{subarray}}a_k\bxi^k}
                  =1+\sum_{\begin{subarray}{c}
          k\in\Z^d_{\ge0}\\[0.4ex]|k|=2,4,6,\dots \end{subarray}}a_k''\bxi^k,\qquad a_k',a_k''\in\R,\ \ \bxi\in\R^d.
\end{multline}
By Lemma~\ref{Lemma_MuBounded}, $\mu(\bxi)$ is bounded. Since expansion~\eqref{MuExpansion} includes only even powers,
it follows that there exist constants $C',C''>0$ such that, for all $\bxi\in\R^d$,
we have estimations~\eqref{MuEstimations}.
\end{proof}

In the following theorem,
the convergence of infinite product~\eqref{CalF} and
continuity of the function ${M}(\bxi)$ are considered.

\begin{theorem}\label{Theorem_OnMContinuity}
Let $\mu(\bs\xi)$ be given by~\eqref{F} and ${M}(\bxi)$ be given
by~\eqref{CalF}. Suppose ${G}(\bxi)$ is of the form~\eqref{m_1},
then infinite product~\eqref{CalF} converges absolutely and
uniformly on any compact set. Moreover, the function $M(\bxi)$ is
continuous on $\R^d$.
\end{theorem}
\begin{proof}
From~\eqref{MuEstimations}, we have
$$
  |\mu(\bxi)-1|\le CP(\bxi),\qquad C>0,\ \forall\bxi\in\R^d.
$$
Using the previous inequality and property~\eqref{QuadrFormInvariance}, we see that
the series $\sum_{j=0}^\infty\left|\mu\left(\left(A^{-T}\right)^j\bxi\right)-1\right|$
uniformly converges on any compact set. Thus infinite product~\eqref{CalF} also uniformly converges on any
compact set.

Moreover, since $\mu(\bxi)$ is continuous
and the partial products converge uniformly to ${M}(\bxi)$
on compact sets; the function ${M}(\bxi)$ is continuous on $\R^d$.
\end{proof}

Below we present an upper
estimate of $M(\bxi)$ at the infinity.
In fact, this estimate is similar to the brute force estimates of
the smoothness of compactly supported univariate wavelets~\cite{Dau1988,Dau}.

\begin{theorem}\label{Theorem_CalMuEstimations}
Let the functions $\mu(\bxi)$, ${M}(\bs\xi)$ satisfy the conditions
of Theorem~\ref{Theorem_OnMContinuity}.
Suppose
\begin{equation}\label{q}
  {\cB}:=\sup_{\bxi\in\R^d}\mu(\bxi),
\end{equation}
then
we have
\begin{equation}\label{HatPhiDecay}
  {M}(\bxi)  \le C\left(1+|\bxi|\right)^{d\log_q{\cB}},\qquad C>0,\ \forall\bxi\in\R^d,
\end{equation}
where $q=|\det A|$.
\end{theorem}
\begin{proof}
By ${\cL}_j$ denote an ellipsoid
$$
  {\cL}_j := \set{\bxi\in\R^d}{P(\bxi)\le q^{2j/d}}.
$$
Using~\eqref{MuEstimations}, we can estimate the function $\mu(\bxi)$ as follows
$$
  \mu(\bxi)\le 1+CP(\bxi)\le \exp\left[CP(\bxi)\right],\qquad \forall\bxi\in\R^d;
$$
and,
using property~\eqref{QuadrFormInvariance}, we have
\begin{equation}\label{ProductMuEstimation}
  \sup_{\bxi\in{\cL}_0}\prod_{j=0}^\infty\mu\left(\left(A^{-T}\right)^j\bxi\right)\le
     \sup_{\bxi\in{\cL}_0}\prod_{j=0}^\infty\exp\left[Cq^{-2j/d}P(\bxi)\right]=\exp\left(C\frac{q^{2/d}}
            {1-q^{2/d}}\right).
\end{equation}

Suppose $\bxi\not\in{\cL}_0$, then there exists a number $J\ge0$ such that
$\bxi\in{\cL}_{J+1}\setminus{\cL}_{J}$. Thus,
\begin{multline*}
 {M}(\bxi)=\prod_{j=0}^\infty\mu\left(\left(A^{-T}\right)^j\bxi\right)
   = \prod_{j=0}^J\mu\left(\left(A^{-T}\right)^j\bxi\right)
       \prod_{j={J+1}}^\infty\mu\left(\left(A^{-T}\right)^j\bxi\right)\\
         = \prod_{j=0}^J\mu\left(\left(A^{-T}\right)^j\bxi\right)
             \prod_{j={0}}^\infty\mu\left(\left(A^{-T}\right)^j
                      \left(A^{-T}\right)^{J+1}\bxi\right).
\end{multline*}
Since $\left(A^{-T}\right)^{J+1}\bxi\in{\cL}_0$; by~\eqref{ProductMuEstimation} and~\eqref{q},
we have ${M}(\bxi)\le {\cB}^{J+1}\exp\left(C\frac{q^{2/d}}
{1-q^{2/d}}\right)$.
If $\bxi\in{\cL}_{J+1}\setminus{\cL}_{J}$;
then $C_1 q^{\frac{J}d}\le|\bxi|\le C_2 q^{\frac{J+1}d}$, $C_1,C_2>0$.
Consequently we have estimation~\eqref{HatPhiDecay}.
\end{proof}

\begin{remark}\label{RemarkOnBruteForceNon-optimality}
  Note that the brute force estimations are very non-optimal. But, in the present paper,
we shall not consider better estimations.
\end{remark}


On the other hand,
if $\cB$ is close to $1$ (recall that $\cB\ge1$), then we can suppose that the function
$M(\bxi)$ decays
as $|\bxi|\to\infty$.
Nevertheless below we show
that ${M}(\bxi)$
does not decay at the infinity.


\begin{theorem}\label{Theorem_UnvanishingOfMu}
Let the functions $\mu(\bxi)$, ${M}(\bs\xi)$ satisfy the conditions
of Theorem~\ref{Theorem_OnMContinuity}.
Let $\bk\in\Z^d$ and $\bs s\in\cS(A^T)\setminus\{\bs 0\}$.
Then, for an arbitrary large $r>\left|2\pi(\bk+\bss)\right|$,
there exists a point $\bs\xi\in\R^d$, $|\bs\xi|>r$, such that ${M}(\bs\xi)
={M}(2\pi (\bs k+\bs s))>0$.
\end{theorem}
\begin{proof}
For some $J\in\N$, $\bk\in\Z^d$, and $\bs s\in\cS(A^T)\setminus\{\bs 0\}$ ,
consider the function ${M}(2\pi\cdot)$ at the point $\left(A^T\right)^{J}(\bk+\bs s)$.
Using~\eqref{Mu=1}, we have
\begin{multline*}
  {M}\left(2\pi\left(A^T\right)^{J}(\bk+\bs s)\right)
    =\prod_{j=0}^{\infty}\mu\left(2\pi \left(A^{T}\right)^{J-j}(\bk+\bs s)\right)\\
    =\prod_{j=0}^{J-1}\mu\left(2\pi \left(A^{T}\right)^{J-j}(\bk+\bs s)\right)
        \prod_{j=J}^{\infty}\mu\left(2\pi \left(A^{T}\right)^{J-j}(\bk+\bs s)\right)\\
    =\prod_{j=0}^{\infty}\mu\left(2\pi \left(A^{T}\right)^{-j}(\bk+\bs s)\right)
    ={M}\left(2\pi (\bk+\bs s)\right).
\end{multline*}
By~\eqref{DilationMatrixProperty},
for an arbitrary large $r>\left|2\pi(\bk+\bss)\right|$, there exists a number $J\in\N$ such that
$2\pi\left|\left(A^T\right)^{J}(\bk+\bss)\right|>r$ and
${M}\left(2\pi\left(A^T\right)^{J}(\bk+\bss)\right)
={M}\left(2\pi (\bk+\bss)\right)>0$.
\end{proof}

Below we present a corollary of Proposition~\ref{PropositionMyDigit}.

\begin{corollary}\label{Corollary_OmegaSet}
Let $A$ be a $d\times d$ dilation matrix. Consider the set
\begin{equation*}
  \Omega_{\bk,\bs s}:=\set{A^j(\bs k+\bs s)}{j\in\N},\qquad \bs k\in\Z^d,\ \ \bs s\in{\cS}(A)
     \setminus\{\bs 0\};
\end{equation*}
then we have
$\Z^d\setminus\{\bs 0\}=\bigcup_{
\begin{subarray}{c}
          \bk\in\Z^d\\[0.2ex]\bs s\in{\cS}(A)\setminus\{\bs 0\}\end{subarray}}\Omega_{\bk,\bs s}$ and
$\Omega_{\bk,\bs s}\cap\Omega_{\bk',\bs s'}=\emptyset$ if $\bs \bk\ne\bk'$ or $\bs s\ne\bs s'$.
\end{corollary}
The proof is left to the reader.

\begin{remark}
By Corollary~\ref{Corollary_OmegaSet},
Theorem~\ref{Theorem_UnvanishingOfMu} guarantees
that
the function ${M}(\bxi)$ does not decay at all (including infinitely distant) points of the lattice
$2\pi\Z^d$.
Whereas $G(\bxi)$ vanishes at the same points $2\pi\bk$, $\bk\in\Z^d$. Consequently
we suppose that $\hat\phi(\bxi)$ can decay faster than
$1/|\bxi|^{2}$ as $|\bxi|\to\infty$.
In any case, Theorem~\ref{Theorem_UnvanishingOfMu}
states that the function $\cF^{-1}M$ can be interpreted only as a
distribution.
\end{remark}


In the next subsection, some corollaries of the properties of the function $M$ are presented.

\subsection{Corollaries}

Since any elliptic scaling function is the convolution
of an elliptic spline with a distribution; we have obvious corollaries of the propositions
of the previous subsection.
\begin{corollary}
If a mask $m_0$ is given by formula~\eqref{m_0Definition}, where
$G(\bxi)$ is given by~\eqref{m_1}; then $\hat\phi(\bxi)$ is
continuous on $\R^d$.
\end{corollary}

\begin{corollary}
Under the conditions of Theorem~\ref{Theorem_CalMuEstimations},
using formulas~\eqref{HatPhi1},~\eqref{HatPhiDecay}, we have
the following estimate
\begin{equation}\label{HatPhiDecayAgain}
  \left|\hat\phi(\bxi)\right| \le C\left(1+|\bxi|\right)^{d\log_{q}{\cB}-2},\qquad\forall\bxi\in\R^d,
\end{equation}
where q:=$|\det A|$ and $\cB$ is given by~\eqref{q}.
\end{corollary}


\section{Elliptic Scaling Function of an Arbitrary Order}\label{Section_ArbitraryOrderFunctions}

An elliptic scaling function of arbitrary order $m>1$ can be defined (in the Fourier domain) by the typical
for the B-splines manner
\begin{equation}\label{HatPhiM}
  \hat\phi^m(\bxi):=\left(\hat\phi(\bxi)\right)^m
    =\frac{\left({G}(\bxi)\right)^m}{\left(P(\bxi)\right)^m}\left({M}(\bxi)\right)^m,
\end{equation}
where ${G}(\bxi)$, ${M}(\bxi)$, and $P(\bxi)$ correspond to the function $\phi$,
see~\eqref{HatPhi1}.
Moreover, for function $\phi^m$,
the mask is $\left(m_0(\bxi)\right)^m$,
where $m_0(\bxi)$ is the mask corresponding to $\phi$.
In other words, the elliptic scaling function of the $m$th, $m>1$, order is defined
as follows
\begin{equation}\label{PhiM}
  \phi^m:=\phi\ast\phi^{m-1},  \qquad\mbox{where}\quad \phi^1:=\phi.
\end{equation}


\subsection{Properties of $\phi^m$}

Now we consider some customary properties of the functions $\phi^m$.
Note that these properties are similar to the properties of
the univariate
cardinal B-splines.

\subsubsection{Compact support}

By definition~\eqref{PhiM}, and since any elliptic scaling function of the first order
is compactly supported; the elliptic scaling function $\phi^m$, $m\ge1$, has a compact support.

\subsubsection{Riesz bases formation}

Here we shall use ideas of the proof from the paper~\cite{VilleBluUnser}.

As is well known, to determine the Riesz bounds, we can
consider the function
\begin{equation}\label{RieszSeries}
  \sum_{\bk\in\Z^d}\left|\hat\phi^m(\bxi+2\pi\bk)\right|^2.
\end{equation}
Note that, using~\eqref{HatPhiDecayAgain} and~\eqref{HatPhiM}, we
have $\left|\hat\phi^m(\bxi)\right| \le
C\left(1+|\bxi|\right)^{m\left(d\log_q{\cB}-2\right)}$,
$\forall\bxi\in\R^d$, where $q:=|\det A|$ and $\cB$ is given
by~\eqref{q}.
If $m\left(d\log_{|\det
A|}{\cB}-2\right)<-d/2$, i.\,e., if
\begin{equation}\label{RieszBasisCondition}
  {\cB}<|\det A|^{\frac{2}{d}-\frac{1}{2m}},
\end{equation}
then series~\eqref{RieszSeries} is convergent
and an upper Riesz bound obviously exists. Note that, increasing the order $m$,
estimation~\eqref{RieszBasisCondition} cannot be made weaker than $\cB<|\det A|^\frac{2}{d}$.

On the other hand, the existence of a lower bound follows from the fact
that $\left|\hat\phi^m(\bxi)\right|^2$ does not vanish in
$[-\pi,\pi]^d$ (because ${G}(\bxi)/P(\bxi)$ does not vanish in
$[-\pi,\pi]^d$).

\begin{remark}
Note that condition~\eqref{RieszBasisCondition} is also
a sufficient condition on $\phi^m$ to belong to $L^2(\R^d)$.
\end{remark}

In fact, $\cB$ is determined completely by the dilation matrix
(by the similarity transformation matrix $Q$).
Note also that, since the brute force estimation is very non-optimal, we suppose
that actually the restrictions on $\cB$
can be essentially weakened. This will be discussed elsewhere.

\subsubsection{Partition of unity}

By Lemma~\ref{LemmaZerosG},
the partition of unity by integer shifts of the scaling function $\phi^m$
is obvious.

\subsubsection{Polynomials representation}

From Lemma~\ref{LemmaZerosG} and formulas~\eqref{HatPhi1},~\eqref{HatPhiM},
we see that $\phi^m$ satisfies the Strang--Fix conditions of $2m-1$ order:
$D^{n}\hat\phi^m(2\pi\bk)=0$, $\forall\bk\in\Z^d\setminus\{\bs 0\}$, $\forall n\in\Z^d_{\ge0}$, $0\le|n|\le 2m-1$.
Thus the span of integer shifts of any elliptic scaling function $\phi^m$ contains {\em all} the polynomials
of total degree $\le 2m-1$.

Nevertheless
the elliptic scaling functions supply an example of the functions that
generalize the classical Strang-Fix conditions and can represent polynomials
of higher degree.

In more detail,
in the paper~\cite{DahmenMicchelli}, W.\;Dahmen and Ch.\;Micchelli introduced a space of polynomials
\begin{equation}\label{SDefinition}
  {\cV}:=\set{p}{p\in\Pi,\,\left(p(D)\hat f\right)(2\pi \bk)=0,\ \bk\in\Z^d\setminus\{0\}},
\end{equation}
where $\Pi$ is the space of {\em all} polynomials on $\R^d$,
$p(D)$ is the differential operator induced by $p$,
and $f$ is a compactly supported function (that, for example, belongs to $S'(\R^d)$).
In the paper~\cite{DahmenMicchelli}, see Proposition~2.1.,
it has been proved that if $f$ satisfies conditions~\eqref{SDefinition} and $\hat f(\bs 0)\ne0$,
then the span of integer shifts of $f$ contains
the {\em affinely-invariant}, i.\,e., scale- and shift-invariant, subspace $\cV_{\rm aff}$
of the space $\cV$.
Namely,
$$
  \sum_{\bk\in\Z^d}p(\bk)f(\bx-\bk) = \hat f(\bs 0)p(\bx)+p^*(\bx),\quad\mbox{where }\ p\in\cV_{\rm aff}\subset\cV,\ \
      \deg p^*<\deg p.
$$
However, in the paper~\cite{deBoor} of
C.~de~Boor, a generalization on not scale-invariant (only shift-invariant) polynomial
spaces was considered.

The well-known functions that define nontrivial spaces $\cV$ (i.\,e., $\cV\subsetneqq\Pi_N$,
where $N$ is the order of the Strang--Fix conditions) are
box-splines.
The elliptic scaling functions supply another example of the functions that define
nontrivial spaces $\cV$.
Below we
state the theorem.
\begin{theorem}~\label{Theorem_NullSpaceRepresentation}
  Let an elliptic scaling function $\phi^m$ be compactly supported and given by~\eqref{HatPhiM},
where $\hat\phi$ is of the form~\eqref{HatPhi1}. Then algebraic polynomials
represented by $\phi^m$ belong to the null-space of the elliptic differential operator
whose symbol is $\left(P(\bxi)\right)^m$.
\end{theorem}
The proof is omitted.
And, to prove the theorem, we refer the reader to Ref.~\cite{Z_ConstrApprox_NDStrFix}, where the proof is based
essentially on the invariance of the quadratic form $P(\bxi)$ under transformation by
the dilation matrix $A^T$; and the proof can be performed even if the explicit form of the Fourier
transform of the scaling function is not known.

Note that, in the case of the elliptic scaling functions, the shift-invariant subspace of
the space $\cV$ is also scale-invariant.
Note also that $\cV$ corresponding to $\phi^m$ contains polynomials of degree $\le 2m+1$.

\subsubsection{Order of approximation}

Let a function $f(\bx)$ and its derivatives up to order $2m$ belong to $L^2(\R^d)$.
Then the function $\phi^m$ supplies the $(2m)$-order of approximation if
$$
  \inf_{c(\bk)}\left|f(\bx)-\sum_{\bk\in\Z^d}c(\bk)\phi^m(\bx/h-\bk)\right|_{L^2}\le Ch^{2m},
$$
where the constant $C$ does not depend on $h$,
see Refs.~\cite{BluUnser,BoorDevoreRon,VilleBluUnser}.

By Lemma~\ref{LemmaZerosG} and formulas~\eqref{HatPhi1},~\eqref{HatPhiM}, we have
$$
  \forall\bk\in\Z^d\setminus\{\bs 0\}:\quad\hat\phi^m(\bxi+2\pi\bk) = O\left(|\bxi|^{2m}\right)
  \quad\mbox{as }\bxi\to\bs 0;
$$
that implies the $(2m)$-order approximation of $\phi^m$.

\subsubsection{Convolution relation}

The convolution relation $\phi^{m_1+m_2}=\phi^{m_1}\ast\phi^{m_2}$ directly follows
from definitions~\eqref{HatPhiM},~\eqref{PhiM}.

\subsubsection{Total positivity}

Let us recall that if a function $f$ is totally positive, then the function satisfies the following
conditions
$$
  \forall \bx_1,\dots,\bx_k\in\R^d,\ \ \forall \xi_1,\dots, \xi_k\in\C,\quad
    \sum_{i,j=1}^k \xi_i\cc \xi_j f(\bx_i-\bx_j)\ge0.
$$
It is known, see for example Ref.~\cite{GelfandVilenkin}, that the previous conditions are equivalent to the
condition: $\hat f(\bxi)\ge0$, $\forall \bxi\in\R^d$.

From Lemmas~\ref{LemmaZerosG},~\ref{Lemma_MuBounded}, it follows that $\hat\phi^m(\bxi)\ge0$,
$\forall\bxi\in\R^d$. So
any elliptic scaling function $\phi^m$ is totally positive.

In the next subsection, we consider some differential properties of the elliptic scaling functions.

\subsection{Differential properties}

It is well known that the B-splines (univariate and multivariate polyharmonic)
satisfy some differential (and integral) relations. Moreover, the B-splines
are built by some simple (differentiable) functions;
in the univariate B-splines case, it is algebraic polynomials.
For the elliptic scaling functions, the situation is more complicated.
Nevertheless, for any elliptic scaling function, we can introduce a (not necessarily differential)
operator such that the scaling function will satisfy some (similar to differential) relations with
this operator.

In detail, consider an elliptic scaling function $\phi^m$, see~\eqref{HatPhiM}, and
define the symbol of an operator as
\begin{equation}\label{SymbolMyOperator}
  \widehat{\Delta^\sharp}(\bxi):=\dfrac{P(\bxi)}{{M}(\bxi)}.
\end{equation}

Now we can present some relations with this operator. But before we must introduce
notation.

Let $\cD_j^2$, $j=1,\dots,d$, be a finite-difference operator such that, for any function $f(\bx)$, $\bx\in\R^d$,
we have $\cD_j^2f:=f(\cdot-\bs e_j)-2f+f(\cdot+\bs e_j)$; and let the operator $\cD_{ij}:=\cD_i\cD_j$,
$i\ne j$, be defined as
$\cD_{ij}f:=\dfrac{1}{4}(f(\cdot-\bs e_i-\bs e_j)$ $+f(\cdot+\bs e_i+\bs e_j)-f(\cdot-\bs e_i+\bs e_j)
-f(\cdot+\bs e_i-\bs e_j))$.

\begin{remark}
Note that,
since $\cD^2_j$ and $\cD_{ij}$ are the iterations of the different operators: $f\left(\cdot-\frac12\bs e_j\right)
-f\left(\cdot+\frac12\bs e_j\right)$ and $\dfrac12\left(f(\cdot-\bs e_j)-f(\cdot+\bs e_j)\right)$,
respectively; it follows
that, if $i=j$, then the operator $\cD_{ij}$ does not change into the operator $\cD^2_j$.
\end{remark}

\begin{theorem}
Let an elliptic scaling function $\phi^m$ be given
by~\eqref{PhiM}, see also~\eqref{HatPhiM}; and let the symbol of
an operator $\Delta^\sharp$ be given by~\eqref{SymbolMyOperator}.
Then we have
\begin{equation}\label{DeltaSharpKPhiM}
  \left(\Delta^\sharp\right)^k \phi^m= \cG^k \phi^{m-k},\quad k<m,
\end{equation}
where $\cG$ is a finite-difference operator corresponding to trigonometric
polynomial~\eqref{m_1}.
The operator $\cG$ is defined as: $\cG:=-P(\cD_1,\dots,\cD_d)$, where $P$ is
quadratic form~\eqref{AnotherQuadraticForm}.
\end{theorem}
\begin{proof}
Using~\eqref{HatPhiM} and~\eqref{SymbolMyOperator}, the Fourier transform
of the left-hand side of relation~\eqref{DeltaSharpKPhiM}
can be written as
$$
  \left(\dfrac{P(\bxi)}{{M}(\bxi)}\right)^k
     \frac{\left({G}(\bxi)\right)^m}{\left(P(\bxi)\right)^m}\left({M}(\bxi)\right)^m
       = \frac{\left({G}(\bxi)\right)^m}{\left(P(\bxi)\right)^{m-k}}\left({M}(\bxi)\right)^{m-k}
         =\left({G}(\bxi)\right)^k\hat\phi^{m-k}(\bxi).
$$
So, in the $x$-domain, we have $\cG^k\phi^{m-k}$, where $\cG$ is a finite-difference operator corresponding
to trigonometric polynomial~\eqref{m_1}.

Rewrite trigonometric polynomial~\eqref{m_1} in exponent-wise form
\begin{multline*}
  G(\bxi) = -\sum_{1\le i\le d} q_{ii}\left(e^{i\xi_i}+e^{-i\xi_i}-2\right)\\
    -\dfrac12\sum_{\begin{subarray}{c}
          1\le i,j\le d\\[0.4ex]i<j \end{subarray}}q_{ij}\left(e^{i(\xi_i+\xi_j)}+e^{-i(\xi_i+\xi_j)}
             -e^{i(\xi_i-\xi_j)}-e^{-i(\xi_i-\xi_j)} \right).
\end{multline*}
It is obvious that, in the $x$-domain, the previous polynomial corresponds to the
following finite-difference expression
\begin{multline*}
  \cG f=-\sum_{1\le i\le d} q_{ii}\Bigl(f(\cdot-\bs e_j)-2f+f(\cdot+\bs e_j)\Bigr)\\
    -\dfrac12\sum_{\begin{subarray}{c}
          1\le i,j\le d\\[0.4ex]i<j \end{subarray}}q_{ij}\Bigl(f(\cdot-\bs e_i-\bs e_j)
             +f(\cdot+\bs e_i+\bs e_j)-f(\cdot-\bs e_i+\bs e_j)-f(\cdot+\bs e_i-\bs e_j)\Bigr)\\
               =-\sum_{1\le i\le d} q_{ii}\cD_i^2f
    -2\sum_{\begin{subarray}{c}
          1\le i,j\le d\\[0.4ex]i<j \end{subarray}}q_{ij}\cD_{ij}f
            =-P(\cD_1,\dots,\cD_d)f.
\end{multline*}
\end{proof}

\begin{remark}
Note that the finite-difference operator
$\cG:=-P(\cD_1,\dots,\cD_d)$ is
the simplest central-difference approximation of the differential operator
with symbol $P(\bxi)$.
\end{remark}

The Fourier transform of the Green function of the operator
$\left(\Delta^\sharp\right)^m$ (where
$\widehat{\Delta^\sharp}(\bxi)$ is given
by~\eqref{SymbolMyOperator}) is of the form
\begin{equation}\label{HatRho}
  \hat\rho(\bxi):=\left(\dfrac{{M}(\bxi)}{P(\bxi)}\right)^m,\qquad\bxi\in\R^d.
\end{equation}
Indeed,
in the Fourier domain, we have
$$
  \cF\left(\left(\Delta^\sharp\right)^m\rho\right)(\bxi)
    = \left(\widehat{\Delta^\sharp}(\bxi)\right)^m\hat\rho(\bxi) \equiv 1.
$$
Thus, in the $x$-domain, we obtain: $\left(\Delta^\sharp\right)^m\rho(\bx)=\delta(\bx)$.
And we can formulate the following proposition.

\begin{proposition}
Any elliptic scaling function $\phi^m$,
see~\eqref{HatPhiM},~\eqref{PhiM}, is a linear combination of
shifted versions of the Green function whose Fourier transform is
defined by~\eqref{HatRho}. Moreover, we have the relation
$$
  \left(\Delta^\sharp\right)^m \phi^m(\bx) = 0,\quad \forall \bx\in\R^d\setminus\Z^d.
$$
\end{proposition}

In the present paper, we shall not study the operators $\left(\Delta^\sharp\right)^m$ in more detail.
In particular, we do not verify wether the operator defined by~\eqref{SymbolMyOperator} is
pseudo-differential. However we can state that the operators $\left(\Delta^\sharp\right)^m$ are elliptic
operators.
Note that the action of the operators $\left(\Delta^\sharp\right)^m$ can be explicitly obtained for some types
of functions, for example, algebraic and trigonometric polynomials.
Here note only that the algebraic and trigonometric polynomials
under action of the operator $\left(\Delta^\sharp\right)^m$
are transformed to polynomials of lower degree:
$\left(\Delta^\sharp\right)^m p=\Delta^m p+\mbox{lower order terms}$,
where $p$ is a polynomial and $\Delta$ is the elliptic operator corresponding to the quadratic form $P(\bxi)$.
Moreover,
the elliptic scaling functions can supply the compactly supported wavelets adapted to
operators~\eqref{SymbolMyOperator}, see for example Refs.~\cite{DW,LPLTch,Z_IJWMIP_OperAdaptWavelets}.
This will be the object of another paper.

Finally note that, using definition~\eqref{SymbolMyOperator},
the Fourier transform of the elliptic scaling function
can be presented as
$$
  \hat\phi^m(\bxi)=\left(\dfrac{G(\bxi)}{\widehat{\Delta^\sharp}(\bxi)}\right)^m,
$$
where $G(\bxi)$ is a
trigonometric polynomial
such that its analytic continuation on $\C^d$ cancels {\em all} zeros of the analytic continuation
of the symbol of the operator $\Delta^\sharp$.
Actually the analytic continuation $\widehat{\Delta^\sharp}(\bs z)$,
$\bs z\in\C^d$, has only one zero at the origin.
This will be discussed elsewhere.


\subsection{Univariate case}


The univariate cardinal B-splines correspond
to the presented approach. Indeed, we see that, in the univariate case,
the dilation `matrix' is number $2$ and $q:=2$.
Thus $P(\xi)=\xi^2$, $\xi\in\R$; consequently, by~\eqref{m_0Definition},~\eqref{m_1}, we have
$$
  {G}(\xi):=4\sin^2\dfrac\xi2=2(1-\cos\xi)\qquad\mbox{and}\qquad
    m_0(\xi):=\frac{{G}(\xi+\pi)}{{G}(\pi)}=\frac{1+\cos\xi}{2}.
$$
Using~\eqref{F},~\eqref{HatPhi1},
we get
\begin{align*}
  \mu(\xi)&=\frac{2(1+\cos\xi/2)(1-\cos\xi/2)}{1-\cos\xi}\equiv1,\\
  \hat\phi(\xi)&=\frac{2(1-\cos\xi)}{\xi^2}=\frac{\sin^2\xi/2}{(\xi/2)^2}=:\hat B_{1}(\xi).
\end{align*}
Hence the univariate elliptic scaling function $\phi$ is the cardinal B-spline of degree 1.
Moreover, the univariate scaling function $\phi^m$ is the cardinal
B-spline of {\em odd }degree $2m-1$ and vice versa.

We can say that the univariate cardinal B-splines satisfy the properties of
Theorem~\ref{Theorem_NullSpaceRepresentation}.
In fact, the B-spline $B_{2m-1}(x)$, $x\in\R$, can be considered as the elliptic scaling function
$\phi^m$ and $B_{2m-1}$ represents {\em all} the polynomials of degree $\le 2m-1$. In the univariate case,
all the polynomials of degree $\le 2m-1$ belong to the null-space of the operator $\dfrac{d^{2m}}{dx^{2m}}$,
but there are not polynomials of degree greater than $2m-1$ that belong to $\ker\dfrac{d^{2m}}{dx^{2m}}$.


\section{Examples}\label{Section_Examples}

In this section, we present several bivariate isotropic dilation matrices and
consider some properties of
the corresponding elliptic scaling functions of the first order.

\subsection{$\det A=2$}

In the bivariate case, the dilation matrix $A$ with determinant $2$ can be isotropic if and only if
$0\le\Tr A\le2$, see for example Ref.~\cite{Z_SIAM_DilMatrs}.

\subsubsection{$\Tr A=2$}\label{Subsubsection_QuincunxMatrix}

Consider a quincunx dilation matrix
\begin{equation}\label{Pi/4_DilationMatrix}
  A_1:=
  \begin{pmatrix}
    1 & -1 \\
    1 & 1 \\
  \end{pmatrix}.
\end{equation}
For $A_1$, $\cS:=\left\{\bs 0, (1/2,1/2)\right\}$; and
the matrix $A_1$ is isotropic.
Thus the matrix $A_1^{-T}$
can be presented in the form~\eqref{MatrixBDecomposition}, where $U$
is a rotation matrix by the angle $\pi/4$ and $Q$ is the identity matrix.
Hence the quadratic form $P(\bxi)$ is $|\bxi|^2$, $\bxi\in\R^2$;
consequently the corresponding differential operator is
the Laplace operator.
The trigonometric polynomial $G(\xi_1,\xi_2):=4\left(\sin^2(\xi_1/2)+\sin^2(\xi_2/2)\right)$
and the mask is of the form
\begin{equation}\label{Pi/4_Mask}
  m_{0,1}(\xi_1,\xi_2):=\frac12+\frac14\cos\xi_1+\frac14\cos\xi_2,
\end{equation}
and $\cB$ (given by~\eqref{q}) is equal to 1.

Matrix~\eqref{Pi/4_DilationMatrix} and the corresponding
scaling function, denoted by $\phi_1$, are `very nice' for several reasons.

\begin{itemize}
\item The factorization of the dilation matrix is very simple.
\item The matrix is a quincunx dilation matrix, i.\,e., it defines the quincunx lattice.
\item The quadratic form is $|\bxi|^2$, $\bxi\in\R^2$, which corresponds to the Laplace operator.
\item Since $\cB=1$, it follows that $\left|\hat\phi_1(\bxi)\right|\le C(1+|\bxi|)^{-2}$, $\forall\bxi\in\R^2$.
Consequently $\phi_1$ forms a Riesz basis and belongs to $L^2\left(\R^2\right)$.
\item The elliptic spline\\
$$
  \cF^{-1}\left(\dfrac{G(\bxi)}{P(\bxi)}\right)(\bx)
    =\cF^{-1}\left(\dfrac{4\left(\sin^2(\xi_1/2)
      +\sin^2(\xi_2/2)\right)}{\xi_1^2+\xi_2^2}\right)(\bx)
$$
is the
polyharmonic B-spline $B_1$, see~\eqref{PolyharmonicBSpline}.
\item The elliptic operator with symbol $\widehat{\Delta^\sharp}(\bxi)
=\dfrac{|\bxi|^2}   {M(\bxi)}$ is uniformly elliptic.
\item The scaling function is an interpolating scaling function,
i.\,e., $\phi_1(\bk)=\delta_{\bk\bs 0}$, $\bk\in\Z^2$. (Because
$m_{0,1}(\xi_1,\xi_2)+m_{0,1}(\xi_1+\pi,\xi_2+\pi)\equiv1$.)
\item The scaling function $\phi_1(\bx)\ge0$
for all $\bx = \left(\frac{k_1}{2^{j_1}},\frac{k_2}{2^{j_2}}\right)$, $k_1,k_2\in\Z$, $j_1,j_2\in\Z_{\ge0}$
(because all the coefficients $h_\bk$ in the refinement relation, see~\eqref{ScalingRelation}, are not negative);
and we can suppose that $\phi_1(\bx)\ge0$, $\forall \bx\in\R^2$.
\item The more traditional quincunx dilation matrix
\begin{equation}\label{QuincunxMatrix}
  \te A_1
  :=\begin{pmatrix}
      1 & 1 \\
      1 & -1 \\
    \end{pmatrix}
\end{equation}
gives the same scaling function. (Because mask~\eqref{Pi/4_Mask} is even and
invariant with respect to the change of the variables.)
\end{itemize}

\begin{remark}
Mask~\eqref{Pi/4_Mask}
has been considered in the book~\cite{NPS} in the context of dual
masks construction (for the quincunx dilation matrix case).
\end{remark}

\subsubsection{$\Tr A=1$}

In this subsection, we consider two very close matrices, but the corresponding elliptic
scaling functions have different properties
\begin{equation}\label{DilationMatricesTr1}
      A_2
      :=\begin{pmatrix}
          0 & -2 \\
          1 & 1 \\
        \end{pmatrix},\qquad\qquad
    A_3
      :=\begin{pmatrix}
          1 & -2 \\
          1 & 0 \\
        \end{pmatrix}.
\end{equation}
Since the determinants and traces of dilation matrices~\eqref{DilationMatricesTr1} are equal;
both the matrices are similar to
the same orthogonal (actually, rotation) matrix, see for example Ref.~\cite{Z_SIAM_DilMatrs},
\begin{equation}\label{RotationMatrixTr1}
  U_2=U_3
  :=\begin{pmatrix}
            \frac{1}{2\sqrt2} & -\frac{\sqrt7}{2\sqrt2} \\
            \frac{\sqrt7}{2\sqrt2} & \frac{1}{2\sqrt2} \\
   \end{pmatrix}.
\end{equation}
The squares of the corresponding similarity transformation matrices are
$$
  Q^{2}_2
    :=\begin{pmatrix}
        2 & -\frac12 \\
        -\frac12 & 1 \\
      \end{pmatrix},\qquad\qquad
  Q^{2}_3
    :=\begin{pmatrix}
        2 & \frac12 \\
        \frac12 & 1 \\
      \end{pmatrix},
$$
respectively. Thus the masks are
\begin{align*}
  &m_{0,2}(\xi_1,\xi_2)
       :=\frac14\left(3+2\cos\xi_1-\cos\xi_2+\frac12\sin\xi_1\sin\xi_2\right),\\
  &m_{0,3}(\xi_1,\xi_2):=\frac16\left(3+2\cos\xi_1+\cos\xi_2+\frac12\sin\xi_1\sin\xi_2\right),
\end{align*}
respectively.

Nevertheless, the decays of the Fourier transform of the corresponding elliptic scaling functions
are very different. Indeed, $\cB_2=2$ and $\cB_3=\dfrac{25}{24}$; consequently
$\hat\phi_2(\bxi)=O(1)$ (see Remark~\ref{RemarkOnBruteForceNon-optimality})
and $\left|\hat\phi_3(\bxi)\right|\le C\left(1+|\bxi|\right)^{-1.882}$, $C>0$,
$\forall\bxi\in\R^2$.
Thus the function $\phi_{3}$ forms a Riesz basis and belongs to
$L^2\left(\R^2\right)$; and the function $\phi_{2}$ can be considered very likely only as a distribution.

\begin{remark}
  The angle of rotation $\arccos\dfrac{1}{2\sqrt2}\approx69.2951889^\circ$ performed by rotation
  matrix~\eqref{RotationMatrixTr1} is {\em incommensurable}
to $\pi$, see Ref.~\cite{Z_SIAM_DilMatrs}. So we can suppose that
if the corresponding scaling functions possess some
angular selectivity; then the wavelet transform can detect arbitrary orientated features.
\end{remark}

\subsection{$\det A=4$}

Finally we consider a diagonal dilation matrix without any rotation
(for this matrix, the orthogonal matrix is the identity matrix)
$$
  A_4:=\begin{pmatrix}
    2 & 0 \\
    0 & 2 \\
  \end{pmatrix}.
$$
Nevertheless, for this matrix, it is also possible to obtain the elliptic scaling function.
Indeed, the similarity transformation matrix is the identity matrix, hence
the quadratic form is $|\bxi|^2$, $\bxi\in\R^2$,
and the mask is of the form
\begin{equation*}
   m_{0,4}(\xi_1,\xi_2)=\frac1{16}\left(2+\cos\xi_1+\cos\xi_2\right)\left(2+\cos\xi_1-\cos\xi_2\right)
        \left(2-\cos\xi_1+\cos\xi_2\right).
\end{equation*}
We have $\cB_4:=\dfrac{9}{8}$, consequently the scaling function
$\phi_4$ corresponding to $A_4$ and $m_{0,4}$ satisfies the estimation
$\left|\hat\phi_4(\bxi)\right|\le C(1+|\bxi|)^{-1.8301}$, $C>0$,
$\forall\bxi\in\R^2$. Thus $\phi_4$ belongs to $L^2\left(\R^2\right)$ and forms a Riesz basis.

Note that, in spite of the fact that the matrix $A_4=\te A^2_1$,
where $\te A_1$ is given by~\eqref{QuincunxMatrix}, the mask
$m_{0,4}(\bxi)\not\equiv m_{0,1}(\te A^T_1\bxi)m_{0,1}(\bxi)$,
where $m_{0,1}$ is given by~\eqref{Pi/4_Mask}; and consequently
$\phi_4$ is not the convolution of the scaling functions from
subsection~\ref{Subsubsection_QuincunxMatrix}.

\section*{Conclusion}

In conclusion, we can summarize that the elliptic scaling functions have all the basic properties
of the univariate cardinal B-splines and can be considered as a generalization of the B-splines
to several variables.

\section*{Acknowledgments}
Research was partially supported by RFBR grant No.~12-01-00608-a.

\end{document}